\documentclass[11pt,letterpaper,reqno]{amsart}
\usepackage{amsmath,amsthm,amsfonts,amssymb}
\usepackage[lite]{amsrefs}

\newtheorem{theorem}{Theorem}
\newtheorem{lemma}[theorem]{Lemma}

\newcommand{\vkq}[1]{[ #1 ]}
\newcommand{\vkqq}[1]{\left[ #1 \right]}
\newcommand{\vkqqq}[1]{\big[ #1 \big]}
\newcommand{\vkqqqq}[1]{\Big[ #1 \Big]}

\begin{document}

\title{On a trilinear form related to the Carleson theorem}
\author{Vjekoslav Kova\v{c}}
\address{Department of Mathematics, University of Zagreb, Bijeni\v{c}ka cesta 30, 10000 Zagreb, Croatia}
\email{vjekovac@math.hr}
\thanks{The author is partially supported by the MZOS grant 037-0372790-2799 of the Republic of Croatia}


\begin{abstract}
The main purpose of this short note is to present an adaptation of the multilinear Bellman function technique from \cite{Kov1}
to the time-frequency analysis.
Demeter and Thiele introduced the two-dimensional bilinear Hilbert transform in \cite{DT}
and showed that the Carleson operator can be identified in particular instances of the corresponding trilinear form.
Demeter considered the Walsh model of one such form in \cite{Dem}, in relation to the discussion of the Walsh-Carleson theorem.
We prove boundedness of this trilinear form for a single triple of exponents at the boundary of the previously established range.
\end{abstract}

\maketitle

\section{Introduction}

The bilinear Hilbert transform is associated with the trilinear form
$$ \Lambda_{1}(f_1,f_2,f_3) := \int_{\mathbb{R}}\Big(\mathrm{p.v.}\!\int_{\mathbb{R}}f_1(x-t)f_2(x+t)\frac{dt}{t}\Big) f_3(x)\,dx . $$
It is a fundamental object in multilinear analysis and its boundedness in a certain range of $\mathrm{L}^p$ spaces
was first shown by Lacey and Thiele in a pair of groundbreaking papers \cite{LT1},\cite{LT2}.
Its two-dimensional analogue was defined by Demeter and Thiele in \cite{DT}. It can be written as
$$ \Lambda_{2}(F_1,F_2,F_3) := \int_{\mathbb{R}^2}\Big(\mathrm{p.v.}\!\int_{\mathbb{R}^2}F_1\big((x,y)+A_1(s,t)\big)
\,F_2\big((x,y)+A_2(s,t)\big) \,K(s,t)\,ds dt\Big) F_3(x,y)\,dx dy , $$
where $K$ is a translation-invariant Calder\'{o}n-Zygmund kernel and $A_1,A_2$ are linear maps on $\mathbb{R}^2$.
A classification of these objects with respect to $A_1$ and $A_2$ was given in \cite{DT}, and it was shown that
\begin{equation}\label{eqlpestimate}
|\Lambda_{2}(F_1,F_2,F_3)| \leq C_{p_1,p_2,p_3} \|F_1\|_{\mathrm{L}^{p_1}(\mathbb{R}^2)}
\|F_2\|_{\mathrm{L}^{p_2}(\mathbb{R}^2)} \|F_3\|_{\mathrm{L}^{p_3}(\mathbb{R}^2)}
\end{equation}
when \,$2\!<\!p_1,p_2,p_3\!<\!\infty$\, and \,$\frac{1}{p_1}\!+\!\frac{1}{p_2}\!+\!\frac{1}{p_3}\!=\!1$,\,
for most of the cases, while the remaining case was studied by the author in \cite{Kov2}.
An interesting observation from \cite{DT} is that boundedness of the linearized Carleson operator
$$ (\mathrm{C}_{N}f)(x) := \mathrm{p.v.}\!\int_{\mathbb{R}}f(x+s)e^{i N(x) s}\frac{ds}{s} $$
on $\mathrm{L}^{p}(\mathbb{R})$, $2<p<\infty$, can be deduced from (\ref{eqlpestimate}) when one chooses
the linear maps and the functions appropriately.
Indeed, the Carleson theorem on a.e.\@ convergence of Fourier series of a square integrable function would follow
from (\ref{eqlpestimate}) if the estimate was extended to $p_1=2$.

We turn to the dyadic model (a.k.a.\@ Walsh model) of $\Lambda_2$, where the Calder\'{o}n-Zygmund kernel $K$ will be replaced by
$$ K_{\mathrm{W}}(s,t) := \sum_{k=0}^{\infty} 2^{2k} \mathbf{1}_{[0,2^{-k})}(s) \mathbf{h}_{[0,2^{-k})}(t) . $$
Here $\mathbf{1}_{I}$ denotes the characteristic function of an interval $I$,
while $\mathbf{h}_{I}$ is the $\mathrm{L}^{\infty}$ normalized Haar wavelet,
$$ \mathbf{h}_{I}:=\mathbf{1}_{\mathrm{left}\,\mathrm{half}\,\mathrm{of}\,I}-\mathbf{1}_{\mathrm{right}\,\mathrm{half}\,\mathrm{of}\,I} . $$
For any $x,y\in[0,\infty)$ written in the binary system as
$x=\sum_{j\in\mathbb{Z}}x_j 2^{j}$, $y=\sum_{j\in\mathbb{Z}}y_j 2^{j}$ we set
$$ x\oplus y := \sum_{j\in\mathbb{Z}} \big((x_j + y_j) \ \mathrm{mod}\ 2\big) \ 2^{j} . $$
In that way the binary operation $\oplus\colon[0,\infty)\times[0,\infty)\to[0,\infty)$ is a.e.\@ well-defined.
Demeter considers the following trilinear form in \cite{Dem}:
\begin{equation}\label{eqmainform}
\Lambda_{\mathrm{W}}(F_1,F_2,F_3) := \sum_{k=0}^{\infty} \,2^{2k}
\int_{[0,1)^4}\! F_1(x\oplus s,y\oplus t) \,F_2(x\oplus t,y) \,F_3(x,y)
\mathbf{1}_{[0,2^{-k})}(s) \,\mathbf{h}_{[0,2^{-k})}(t) \,dx dy ds dt .
\end{equation}
The material in Section 8 of \cite{Dem} is a less technical variant of the proof from \cite{DT} and it establishes Estimate
(\ref{eqlpestimate}) for $\Lambda_{\mathrm{W}}$ in the same region of $(p_1,p_2,p_3)$.
In this paper we modestly expand the range of exponents by a single point, namely $(2,4,4)$.
This choice of exponents is interesting because of $p_1=2$ and the analogy with the previously mentioned continuous case.

\begin{theorem}\label{thmmaintheorem}
Trilinear form $\Lambda_{\mathrm{W}}$, defined by \emph{(\ref{eqmainform})} for dyadic step functions $F_1,F_2,F_3$,
satisfies the estimate
$$ |\Lambda_{\mathrm{W}}(F_1,F_2,F_3)| \,\leq\, 7 \,\|F_1\|_{\mathrm{L}^{2}([0,1)^2)}
\|F_2\|_{\mathrm{L}^{4}([0,1)^2)} \|F_3\|_{\mathrm{L}^{4}([0,1)^2)} . $$
\end{theorem}

The emphasis is on the method we use and on simplicity of the proof, avoiding any tree decomposition or tree selection algorithms.
First, we decompose $\Lambda_{\mathrm{W}}$ in the $\frac{3}{2}$-dimensional phase space, invented by Demeter and Thiele
in \cite{DT} for the same purpose.
Then we adapt the multilinear Bellman function technique from \cite{Kov1} to this time-frequency setting.

We have to remark that Theorem \ref{thmmaintheorem} does not imply a.e.\@ convergence of Walsh-Fourier series
of a function in $\mathrm{L}^{2}([0,1))$, originally shown by Billard in \cite{Bil}.
To reprove that result, one would have to establish the same estimate for a trilinear form obtained by
inserting absolute values around each term of the time-frequency decomposition of $\Lambda_{\mathrm{W}}$.
However, such proof would probably require techniques used in \cite{DT},
so it is not very likely that this short self-contained argument would still apply.
On the other hand, $\Lambda_{\mathrm{W}}$ is still a sufficiently nontrivial object and any result on its boundedness almost certainly
requires some time-frequency analysis.

Our approach is mainly inspired by an unpublished proof by F. Nazarov of the $(\mathrm{L}^3,\mathrm{L}^3,\mathrm{L}^3)$ bound
for the one-dimensional bilinear Hilbert transform in characteristic $3$.
There is also some similarity with ``time-frequency swapping'' from the proof of the nonlinear Plancherel inequality
by Muscalu, Tao, and Thiele in \cite{MTT}.
The motivation for writing this note lies in scarcity of examples of proofs where one constructs a Bellman
function that also depends on the frequency, besides the position and the scale as usually.

\section{Time-frequency decomposition of the form}

Let $\mathcal{D}_{+}$ denote the family of \emph{dyadic intervals} in $[0,\infty)$, i.e.
$$ \mathcal{D}_{+} := \big\{ \big[2^{-k}\ell,2^{-k}(\ell+1)\big) \,:\, k,\ell\in\mathbb{Z}, \ \ell\geq 0 \big\} . $$
For each $I\in\mathcal{D}_{+}$ its left half will be denoted $I_0$ and its right half will be denoted $I_1$.
These are also dyadic intervals.
A \emph{dyadic step function} on $\mathbb{R}$ will simply be a finite linear combination of characteristic functions of intervals from $\mathcal{D}_{+}$.
Similarly, a dyadic step function on $\mathbb{R}^2$ means a finite linear combination of characteristic functions of dyadic rectangles,
i.e.\@ rectangles with sides in $\mathcal{D}_{+}$.

Take a dyadic interval $I=\big[2^{-k}\ell,2^{-k}(\ell+1)\big)\in\mathcal{D}_{+}$ and a nonnegative integer $n=\sum_{j=0}^{m}n_j 2^{j}$.
We can define the ($\mathrm{L}^\infty$-normalized) \emph{Walsh wave packet} $w_{I,n}$ directly by
$$ w_{I,n}(x) := (-1)^{\sum_{j=0}^{m} n_j x_{-j-k-1}} \mathbf{1}_{I}(x) , $$
where $x=\sum_{j\in\mathbb{Z}}x_j 2^{j}$.
If we also denote $\Omega=\big[2^{k}n,2^{k}(n+1)\big)$, we can associate the Walsh wave packet to the rectangle $I\times\Omega$ of area $1$,
and write $w_{I\times\Omega}$ instead of $w_{I,n}$.
Immediate properties of this system are:
\begin{itemize}
\item[(W1)]
$w_{I,n_1}(x) \,w_{I,n_2}(x) = w_{I,n_1\oplus n_2}(x)$ \
for $I\in\mathcal{D}_{+}$, for nonnegative integers $n_1,n_2$, and for a.e.\@ $x\in\mathbb{R}$.
\item[(W2)]
$w_{I,n}(x) \,w_{J,n}(y) = w_{I\oplus J,n}(x\oplus y)$ \
for $I,J\in\mathcal{D}_{+}$ such that $|I|=|J|$, for a nonnegative integer $n$, and for almost all $x\in I$, $y\in J$.
Here we denote $I\oplus J:=\big[2^{-k}\ell,2^{-k}(\ell+1)\big)$ if $I=\big[2^{-k}\ell_1,2^{-k}(\ell_1+1)\big)$,
$J=\big[2^{-k}\ell_2,2^{-k}(\ell_2+1)\big)$, and $\ell=\ell_1\oplus\ell_2$.
\item[(W3)]
$\sum_{n=0}^{\infty}\big(|I|^{-1}\!\int_{I}f(t)w_{I,n}(t)dt\big)\,w_{I,n}(x)=f(x)$ \
for $I\in\mathcal{D}_{+}$, for any dyadic step function $f$ supported on $I$, and for a.e.\@ $x\in I$.
This is the Walsh-Fourier expansion of $f$ and the sum is effectively finite, i.e.\@ only finitely many terms are nonzero.
\item[(W4)]
$w_{I\times\Omega_{0}}=w_{I_{0}\times\Omega}+w_{I_{1}\times\Omega}$
\ and \ $w_{I\times\Omega_{1}}=w_{I_{0}\times\Omega}-w_{I_{1}\times\Omega}$ \
for any $I,\Omega\in\mathcal{D}_{+}$ such that $|I| |\Omega|=2$.
\item[(W5)]
$w_{I,0}=\mathbf{1}_{I}$ and $w_{I,1}=\mathbf{h}_{I}$ for any $I\in\mathcal{D}_{+}$.
\end{itemize}
One can consult \cite{Thi} for more details on Walsh wave packets.
The only modification is that we prefer to normalize the system in $\mathrm{L}^\infty$, rather
than $\mathrm{L}^2$, which is more common.

Let us also introduce the notation
$$ [f(x)]_{x,I\times\Omega} = [f(x)]_{x,I,n} := \frac{1}{|I|}\int_{I} f(x) w_{I,n}(x) dx $$
for a locally integrable function $f$ and for $I=\big[2^{-k}\ell,2^{-k}(\ell+1)\big)$, $\Omega=\big[2^{k}n,2^{k}(n+1)\big)$,
$k,\ell,n\in\mathbb{Z}$, \,$\ell,n\geq 0$.
In particular, $[f(x)]_{x,I,0}$ is the ordinary average of $f$ on $I$.
We want to emphasize notationally in which variable we are averaging, because we will be dealing with expressions in several variables.

In the rest of the paper we fix three real-valued dyadic step functions $F_1,F_2,F_3$ supported on $[0,1)^2$.
Take a positive integer $M$ that is large enough so that $F_1,F_2,F_3$ are constant on dyadic squares with sides of length $2^{-M}$.
No estimates will depend on $M$ and we only use it to keep the arguments finite.
Finally, since the estimate we are proving is homogenous, we can normalize the functions by
\begin{equation}\label{eqnormalize}
\|F_1\|_{\mathrm{L}^2} = 1, \quad \|F_2\|_{\mathrm{L}^4} = \|F_3\|_{\mathrm{L}^4} = 1 .
\end{equation}

We begin to decompose $\Lambda_{\mathrm{W}}$ by breaking the integrals over $[0,1)$ in $x$ and $y$
into integrals over dyadic intervals of length $2^{-k}$.
Then we change $\mathbf{h}_{[0,2^{-k})}$ to $w_{[0,2^{-k}),1}$, as in (W5):
\begin{align*}
\Lambda_{\mathrm{W}}(F_1,F_2,F_3) = \sum_{k=0}^{\infty} \,2^{2k}
\!\!\!\!\!\sum_{\substack{I,J\in\mathcal{D}_{+},\ I,J\subseteq[0,1)\\ |I|=|J|=2^{-k}}}
& \int_{\mathbb{R}^4} F_1(x\oplus s,y\oplus t) \,F_2(x\oplus t,y) \,F_3(x,y) \\[-7mm]
& \qquad \mathbf{1}_{I}(x) \mathbf{1}_{J}(y) \mathbf{1}_{[0,2^{-k})}(s) w_{[0,2^{-k}),1}(t) \,dx dy ds dt .
\end{align*}
We substitute $x_1=x\oplus s$, \,$x_2=x\oplus t$, \,$x_3=x$, so that $s=x_1\oplus x_3$ and $t=x_2\oplus x_3$.
Note that translation \,$s\mapsto x\oplus s$\, preserves the Lebesgue measure and that it also preserves dyadic intervals of length $2^{-k}$.
Using these facts together with Property (W2) for $n=1$, we get
\begin{align*}
\Lambda_{\mathrm{W}}(F_1,F_2,F_3) = \sum_{k=0}^{M-1} \,2^{2k}
\!\!\!\!\sum_{\substack{I,J\in\mathcal{D}_{+},\ I,J\subseteq[0,1)\\ |I|=|J|=2^{-k}}}
&\int_{\mathbb{R}^4} F_1(x_1,y\oplus x_2\oplus x_3) F_2(x_2,y) F_3(x_3,y) \\[-7mm]
& \qquad w_{I,0}(x_1) w_{I,1}(x_2) w_{I,1}(x_3) w_{J,0}(y) \,dx_1 dx_2 dx_3 dy .
\end{align*}
By Property (W3) we can decompose
$$ F_1(x_1,y\oplus x_2\oplus x_3)  = \sum_{n=0}^{2^{M}\!|J|-1} \vkqq{F_1(x_1,y_1)}_{y_1,J,n} w_{J,n}(y\oplus x_2\oplus x_3),
\quad x_1,x_2,x_3\in I,\ \ y\in J . $$
Here we recall that $\vkq{F_1(x_1,y_1)}_{y_1,J,n}$ vanishes when $n\geq 2^{M}|J|$,
because of our assumptions on $F_1,F_2,F_3$.
An application of (W1) and (W2) gives
$$ w_{J,n}(y\oplus x_2\oplus x_3) w_{I,0}(x_1) w_{I,1}(x_2) w_{I,1}(x_3) w_{J,0}(y)
= w_{I,0}(x_1) w_{I,n\oplus 1}(x_2) w_{I,n\oplus 1}(x_3) w_{J,n}(y) , $$
and since $|I\times J|=2^{-2k}$, the form becomes
\begin{align}\label{eqmainform2}
\Lambda_{\mathrm{W}}(F_1,F_2,F_3) = \!\!\!\sum_{\substack{I,J\in\mathcal{D}_{+}\\ I,J\subseteq[0,1)\\ |I|=|J|\geq 2^{-M+1}}}
\!\!\!\sum_{0\leq n<2^{M}\!|I|} \ & |I\times J| \,\vkqqq{\vkqq{F_1(x_1,y_1)}_{x_1,I,0}\!}_{y_1,J,n} \nonumber \\[-8mm]
& \times\vkqqq{\vkqq{F_2(x_2,y)}_{x_2,I,n\oplus 1} \vkqq{F_3(x_3,y)}_{x_3,I,n\oplus 1}\!}_{y,J,n} .
\end{align}

\section{Time-frequency Bellman functions}

A rectangular cuboid $I\times J\times \Omega$ formed by $I,J,\Omega\in\mathcal{D}_{+}$ will be called
a \emph{tile} if $|I|=|J|=|\Omega|^{-1}$ and a \emph{multitile} if $|I|=|J|=2|\Omega|^{-1}$.
We will only be interested in tiles and multitiles contained in $[0,1)^2\!\times\![0,2^M)$.
Therefore we define the collections:
\begin{align*}
& \mathcal{T}_{k} :=\big\{I\!\times\!J\!\times\!\Omega\ \textrm{ is a tile }
: \ I,J\subseteq[0,1),\ \Omega\subseteq[0,2^M),\ |I|=2^{-k}\big\}, \quad \mathcal{T} := {\textstyle\bigcup_{k=0}^{M}}\mathcal{T}_{k},\\
& \mathcal{M}_{k} :=\big\{I\!\times\!J\!\times\!\Omega\ \textrm{ is a multitile }
: \ I,J\subseteq[0,1),\ \Omega\subseteq[0,2^M),\ |I|=2^{-k}\big\}, \quad \mathcal{M} := {\textstyle\bigcup_{k=0}^{M-1}}\mathcal{M}_{k}.
\end{align*}
A multitile $P=I\times J\times \Omega$ can be divided ``horizontally'' into four tiles
$$ P_{0,0}:=I_0\times J_0\times \Omega,
\ \ P_{0,1}:=I_0\times J_1\times \Omega,
\ \ P_{1,0}:=I_1\times J_0\times \Omega,
\ \ P_{1,1}:=I_1\times J_1\times \Omega $$
and ``vertically'' into two tiles
$$ P^0:=I\times J\times\Omega_0,
\ \ P^1:=I\times J\times\Omega_1. $$
For any function of tiles $\mathcal{B}\colon\mathcal{T}\to\mathbb{R}$ we define its \emph{first-order difference}
as a function of multitiles $\Box\mathcal{B}\colon\mathcal{M}\to\mathbb{R}$ given by the formula
$$ (\Box\mathcal{B})(P) := \frac{1}{4}\sum_{\alpha,\beta\in\{0,1\}}\!\mathcal{B}(P_{\alpha,\beta})
- \sum_{\gamma\in\{0,1\}}\!\mathcal{B}(P^{\gamma}) . $$

Note that (\ref{eqmainform2}) can now be rewritten as
\begin{equation}\label{eqarewritten}
\Lambda_{\mathrm{W}} = \!\sum_{I\times J\times \Omega\in\mathcal{M}}\!|I\times J|\ \mathcal{A}(I\times J\times \Omega) ,
\end{equation}
where $\mathcal{A}\colon\mathcal{M}\to\mathbb{R}$ is given by
$$ \mathcal{A}(I\!\times\! J\!\times\! \Omega)
:= \sum_{\gamma\in\{0,1\}} \vkqqq{\vkqq{F_1(x,y)}_{x,I,0}}_{y,J\times\Omega_{\gamma}}
\vkqqq{\vkqq{F_2(x,y)}_{x,I\times\Omega_{\gamma\oplus 1}}
\vkqq{F_3(x,y)}_{x,I\times\Omega_{\gamma\oplus 1}}}_{y,J\times\Omega_{\gamma}}. $$
This follows from the fact that $n$ and $n\oplus 1$ are two consecutive nonnegative integers and the smaller one is even,
so the union of
\,$I\times J\times\big[{\textstyle\frac{n}{|I|}},{\textstyle\frac{n+1}{|I|}}\big)$ and
$I\times J\times\big[{\textstyle\frac{n\oplus 1}{|I|}},{\textstyle\frac{(n\oplus 1)+1}{|I|}}\big)$\,
constitutes a multitile from $\mathcal{M}$.

We will be writing $I\times J\times \Omega$ both for a generic tile and for a generic multitile.
However, it will always be clear if the domain of the function in question is $\mathcal{T}$ or $\mathcal{M}$.
In order to control $\mathcal{A}$ we need to introduce several relevant expressions.
Define \,$\mathcal{B}_i\colon\mathcal{T}\to\mathbb{R}$, \,$i=1,2,3,4,5$\, by
{\allowdisplaybreaks\begin{align*}
\mathcal{B}_1(I\!\times\! J\!\times\! \Omega) &
:= \vkqqq{\vkqq{F_1(x,y)}_{x,I,0}}_{y,J\times\Omega}
\vkqqq{\vkqq{F_2(x,y)}_{x,I\times\Omega} \vkqq{F_3(x,y)}_{x,I\times\Omega}}_{y,J\times\Omega}, \\
\mathcal{B}_2(I\!\times\! J\!\times\! \Omega) &
:= \vkqqq{\vkqq{F_1(x,y)}_{x,I,0}}_{y,J\times\Omega}^2, \\
\mathcal{B}_3(I\!\times\! J\!\times\! \Omega) &
:= \vkqqq{\vkqq{F_2(x,y)}_{x,I\times\Omega} \vkqq{F_3(x,y)}_{x,I\times\Omega}}_{y,J\times\Omega}^2, \\
\mathcal{B}_4(I\!\times\! J\!\times\! \Omega) &
:= \vkqqq{\vkqq{F_2(x,y)}_{x,I\times\Omega}^2}_{y,J,0}^2, \\
\mathcal{B}_5(I\!\times\! J\!\times\! \Omega) &
:= \vkqqq{\vkqq{F_3(x,y)}_{x,I\times\Omega}^2}_{y,J,0}^2.
\end{align*}}
We remark that the scope of each averaging variable (i.e.\@ $x$ or $y$) is only inside the corresponding bracket.
Finally, denote
$$ \mathcal{B}_{-}:=\mathcal{B}_{1}-\mathcal{B}_{2}-{\textstyle\frac{1}{2}}\mathcal{B}_{3}
-{\textstyle\frac{1}{2}}\mathcal{B}_{4}-{\textstyle\frac{1}{2}}\mathcal{B}_{5}\quad\textrm{and}\quad
\mathcal{B}_{+}:=\mathcal{B}_{1}+\mathcal{B}_{2}+{\textstyle\frac{1}{2}}\mathcal{B}_{3}
+{\textstyle\frac{1}{2}}\mathcal{B}_{4}+{\textstyle\frac{1}{2}}\mathcal{B}_{5} . $$
The key local estimate is contained in the following lemma.

\begin{lemma}\label{lmlemmaineq}
For every \,$I\times J\times \Omega\in\mathcal{M}$ we have
$$ \Box\mathcal{B}_{-}(I\!\times\! J\!\times\! \Omega) \,\leq\, \mathcal{A}(I\!\times\! J\!\times\! \Omega)
\,\leq\, \Box\mathcal{B}_{+}(I\!\times\! J\!\times\! \Omega) . $$
\end{lemma}

\begin{proof}
We need to compute the first-order difference of each $\mathcal{B}_i$, \,$i=1,2,3,4,5$.
Let us fix some \,$I\times J\times \Omega\in\mathcal{M}$.
Using Property (W4) alone, one can easily verify the following identity:
{\allowdisplaybreaks\begin{align*}
& 8\textstyle\sum_{\alpha,\beta\in\{0,1\}} w_{I_{\alpha},0}(x_1) w_{I_{\alpha}\times\Omega}(x_2) w_{I_{\alpha}\times\Omega}(x_3)
w_{J_{\beta}\times\Omega}(y_1) w_{J_{\beta}\times\Omega}(y_2) \\
& \quad - \textstyle\sum_{\gamma\in\{0,1\}} w_{I,0}(x_1) w_{I\times\Omega_\gamma}(x_2) w_{I\times\Omega_\gamma}(x_3)
w_{J\times\Omega_\gamma}(y_1) w_{J\times\Omega_\gamma}(y_2) \\
& = \textstyle\sum_{\gamma\in\{0,1\}} w_{I,0}(x_1) w_{I\times\Omega_{\gamma\oplus 1}}(x_2) w_{I\times\Omega_{\gamma\oplus 1}}(x_3)
w_{J\times\Omega_\gamma}(y_1) w_{J\times\Omega_\gamma}(y_2) \\
& \quad + \textstyle\sum_{\gamma,\delta\in\{0,1\}} w_{I,1}(x_1) w_{I\times\Omega_{\delta}}(x_2) w_{I\times\Omega_{\delta\oplus 1}}(x_3)
w_{J\times\Omega_\gamma}(y_1) w_{J\times\Omega_\gamma}(y_2) .
\end{align*}}
If we multiply it by $F_1(x_1,y_1) F_2(x_2,y_2) F_3(x_3,y_2)$, integrate in variables $x_1,x_2,x_3,y_1,y_2$ over $\mathbb{R}^5$,
and divide the obtained expression by $|I|^3 |J|^2$, we are left with
$$ (\Box\mathcal{B}_1)(I\!\times\! J\!\times\! \Omega) = \mathcal{A}(I\!\times\! J\!\times\! \Omega)
+ \underbrace{\sum_{\gamma,\delta\in\{0,1\}}
\!\!\vkqqq{\vkqq{F_1(x,y)}_{x,I,1}}_{y,J\times\Omega_{\gamma}} \vkqqq{\vkqq{F_2(x,y)}_{x,I\times\Omega_{\delta}}
\vkqq{F_3(x,y)}_{x,I\times\Omega_{\delta\oplus 1}}}_{y,J\times\Omega_{\gamma}}}_{\mathcal{A}_1(I\times J\times\Omega)} . $$
Similarly, if we start with
\begin{align*}
& 4\textstyle\sum_{\alpha,\beta\in\{0,1\}} \prod_{i\in\{1,2\}} w_{I_\alpha,0}(x_i) w_{J_{\beta}\times\Omega}(y_i)
- \textstyle\sum_{\gamma\in\{0,1\}} \prod_{i\in\{1,2\}} w_{I,0}(x_i) w_{J\times\Omega_{\gamma}}(y_i) \\
& = \textstyle\sum_{\gamma\in\{0,1\}} \prod_{i\in\{1,2\}} w_{I,1}(x_i) w_{J\times\Omega_{\gamma}}(y_i) ,
\end{align*}
multiply it by $F_1(x_1,y_1) F_1(x_2,y_2)$, integrate in $x_1,x_2,y_1,y_2$, and divide by $|I|^2 |J|^2$, we get
$$ (\Box\mathcal{B}_2)(I\!\times\! J\!\times\! \Omega) = \sum_{\gamma\in\{0,1\}}
\vkqqq{\vkqq{F_1(x,y)}_{x,I,1}}_{y,J\times\Omega_{\gamma}}^2 . $$
In exactly the same way, identity
{\allowdisplaybreaks\begin{align*}
& 16\textstyle\sum_{\alpha,\beta\in\{0,1\}} \big(\textstyle\prod_{i,j\in\{0,1\}} w_{I_{\alpha}\times\Omega}(x_{i,j})\big)
\big(\textstyle\prod_{i\in\{0,1\}} w_{J_\beta\times\Omega}(y_i)\big) \\
& \quad - \textstyle\sum_{\gamma\in\{0,1\}} \big(\textstyle\prod_{i,j\in\{0,1\}} w_{I\times\Omega_\gamma}(x_{i,j})\big)
\big(\textstyle\prod_{i\in\{0,1\}} w_{J\times\Omega_\gamma}(y_i)\big) \\
& = \textstyle\sum_{\gamma,\delta\in\{0,1\}} \big(\textstyle\prod_{i,j\in\{0,1\}} w_{I\times\Omega_{\delta\oplus j}}(x_{i,j})\big)
\big(\textstyle\prod_{i\in\{0,1\}} w_{J\times\Omega_\gamma}(y_i)\big) \\
& \quad + 2\textstyle\sum_{\beta,\gamma,\delta\in\{0,1\}}
\big(\textstyle\prod_{i\in\{0,1\}} w_{I\times\Omega_{\gamma\oplus i}}(x_{i,0})\big)
\big(\textstyle\prod_{i\in\{0,1\}} w_{I\times\Omega_{\delta\oplus i}}(x_{i,1})\big)
\big(\textstyle\prod_{i\in\{0,1\}} w_{J_\beta\times\Omega}(y_i)\big) \\
& \quad + \textstyle\sum_{\gamma\in\{0,1\}} \big(\textstyle\prod_{i,j\in\{0,1\}} w_{I\times\Omega_{\gamma\oplus 1}}(x_{i,j})\big)
\big(\textstyle\prod_{i\in\{0,1\}} w_{J\times\Omega_\gamma}(y_i)\big)
\end{align*}}
multiplied by $F_2(x_{0,0},y_0) F_3(x_{0,1},y_0) F_2(x_{1,0},y_1) F_3(x_{1,1},y_1)$ and integrated in
\,$x_{0,0},x_{0,1},x_{1,0},x_{1,1},y_{0},y_{1}$ gives
{\allowdisplaybreaks\begin{align*}
(\Box\mathcal{B}_3)(I\!\times\! J\!\times\! \Omega) &
= \underbrace{\sum_{\gamma,\delta\in\{0,1\}} \vkqqq{\vkqq{F_2(x,y)}_{x,I\times\Omega_{\delta}}
\vkqq{F_3(x,y)}_{x,I\times\Omega_{\delta\oplus 1}}}_{y,J\times\Omega_{\gamma}}^2}_{\mathcal{A}_2(I\times J\times\Omega)} \\
& \quad + \underbrace{\sum_{\beta,\delta\in\{0,1\}} \prod_{\gamma\in\{0,1\}} \vkqqq{\vkqq{F_2(x,y)}_{x,I\times\Omega_{\gamma}}
\vkqq{F_3(x,y)}_{x,I\times\Omega_{\gamma\oplus\delta}}}_{y,J_{\beta}\times\Omega}}_{\mathcal{A}_3(I\times J\times\Omega)} \\
& \quad + \sum_{\gamma\in\{0,1\}} \underbrace{\vkqqq{\vkqq{F_2(x,y)}_{x,I\times\Omega_{\gamma\oplus 1}}
\vkqq{F_3(x,y)}_{x,I\times\Omega_{\gamma\oplus 1}}}_{y,J\times\Omega_{\gamma}}^2}_{\geq 0} ,
\end{align*}}
while identity
{\allowdisplaybreaks\begin{align*}
& 16\textstyle\sum_{\alpha,\beta\in\{0,1\}} \big(\textstyle\prod_{i,j\in\{0,1\}} w_{I_{\alpha}\times\Omega}(x_{i,j})\big)
\big(\textstyle\prod_{i\in\{0,1\}} w_{J_\beta,0}(y_i)\big) \\
& \quad - \textstyle\sum_{\gamma\in\{0,1\}} \big(\textstyle\prod_{i,j\in\{0,1\}} w_{I\times\Omega_\gamma}(x_{i,j})\big)
\big(\textstyle\prod_{i\in\{0,1\}} w_{J,0}(y_i)\big) \\
& = 2\textstyle\sum_{\beta,\gamma\in\{0,1\}} \big(\textstyle\prod_{i,j\in\{0,1\}} w_{I\times\Omega_{\gamma\oplus i}}(x_{i,j})\big)
\big(\textstyle\prod_{i\in\{0,1\}} w_{J_\beta,0}(y_i)\big) \\
& \quad + \textstyle\sum_{\gamma\in\{0,1\}} \big(\textstyle\prod_{i,j\in\{0,1\}} w_{I\times\Omega_\gamma}(x_{i,j})\big)
\big(\textstyle\prod_{i\in\{0,1\}} w_{J,1}(y_i)\big) \\
& \quad + \textstyle\sum_{\gamma,\delta,\epsilon\in\{0,1\}} \big(\textstyle\prod_{j\in\{0,1\}} w_{I\times\Omega_{\gamma\oplus j}}(x_{0,j})\big)
\big(\textstyle\prod_{j\in\{0,1\}} w_{I\times\Omega_{\delta\oplus j}}(x_{1,j})\big)
\big(\textstyle\prod_{i\in\{0,1\}} w_{J,\epsilon}(y_i)\big)
\end{align*}}
multiplied by $F_2(x_{0,0},y_0) F_2(x_{0,1},y_0) F_2(x_{1,0},y_1) F_2(x_{1,1},y_1)$ implies
{\allowdisplaybreaks\begin{align*}
(\Box\mathcal{B}_4)(I\!\times\! J\!\times\! \Omega) &
= \underbrace{\sum_{\beta\in\{0,1\}} \prod_{\gamma\in\{0,1\}}
\vkqqq{\vkqq{F_2(x,y)}_{x,I\times\Omega_\gamma}^2}_{y,J_{\beta},0}}_{\mathcal{A}_4(I\times J\times\Omega)}
+ \sum_{\gamma\in\{0,1\}} \underbrace{\vkqqq{\vkqq{F_2(x,y)}_{x,I\times\Omega_{\gamma}}^2}_{y,J,1}^2}_{\geq 0}\\
& \quad + 4\sum_{\epsilon\in\{0,1\}} \underbrace{\vkqqq{\vkqq{F_2(x,y)}_{x,I\times\Omega_{0}}
\vkqq{F_2(x,y)}_{x,I\times\Omega_{1}}}_{y,J,\epsilon}^2}_{\geq 0} .
\end{align*}}
The expression for $\Box\mathcal{B}_5$ has $F_2$ replaced with $F_3$:
{\allowdisplaybreaks\begin{align*}
(\Box\mathcal{B}_5)(I\!\times\! J\!\times\! \Omega) &
= \underbrace{\sum_{\beta\in\{0,1\}} \prod_{\gamma\in\{0,1\}}
\vkqqq{\vkqq{F_3(x,y)}_{x,I\times\Omega_\gamma}^2}_{y,J_{\beta},0}}_{\mathcal{A}_5(I\times J\times\Omega)}
+ \sum_{\gamma\in\{0,1\}} \underbrace{\vkqqq{\vkqq{F_3(x,y)}_{x,I\times\Omega_{\gamma}}^2}_{y,J,1}^2}_{\geq 0}\\
& \quad + 4\sum_{\epsilon\in\{0,1\}} \underbrace{\vkqqq{\vkqq{F_3(x,y)}_{x,I\times\Omega_{0}}
\vkqq{F_3(x,y)}_{x,I\times\Omega_{1}}}_{y,J,\epsilon}^2}_{\geq 0} .
\end{align*}}

Several terms that appear in $\Box\mathcal{B}_i$ were denoted $\mathcal{A}_j$,
while some terms were noted to be nonnegative, being squares of real numbers. Consequently,
\begin{equation}\label{eqbineq1}
\Box\mathcal{B}_{3}+\Box\mathcal{B}_{4}+\Box\mathcal{B}_{5}
\geq \mathcal{A}_{2}+\mathcal{A}_{3}+\mathcal{A}_{4}+\mathcal{A}_{5} .
\end{equation}
An immediate application of the inequality $|ab|\leq\frac{1}{2}a^2+\frac{1}{2}b^2$ gives
\begin{equation}\label{eqbineq2}
|\mathcal{A}-\Box\mathcal{B}_1| = |\mathcal{A}_1| \leq \Box\mathcal{B}_2 + {\textstyle\frac{1}{2}}\mathcal{A}_2 .
\end{equation}
We will also need the inequality $|[f(y)]_{y,J_{\beta}\times\Omega}|\leq[|f(y)|]_{y,J_{\beta},0}$ for any real-valued function $f$.
Using the Cauchy-Schwarz inequality we can estimate
\begin{align*}
|\mathcal{A}_3(I\!\times\! J\!\times\! \Omega)| &
\leq \sum_{\beta,\delta\in\{0,1\}} \prod_{\gamma\in\{0,1\}}
\vkqqqq{\big|\!\vkqq{F_2(x,y)}_{x,I\times\Omega_{\gamma}}\!\big|
\,\big|\!\vkqq{F_3(x,y)}_{x,I\times\Omega_{\gamma\oplus\delta}}\!\big|}_{y,J_{\beta},0} \\
& \leq 2\sum_{\beta\in\{0,1\}} \prod_{\gamma\in\{0,1\}}
\vkqqq{\vkqq{F_2(x,y)}_{x,I\times\Omega_\gamma}^{2}}_{y,J_\beta,0}^{1/2}
\vkqqq{\vkqq{F_3(x,y)}_{x,I\times\Omega_\gamma}^{2}}_{y,J_\beta,0}^{1/2} \\
& \leq\sum_{\beta\in\{0,1\}}\Big(\prod_{\gamma\in\{0,1\}}\vkqqq{\vkqq{F_2(x,y)}_{x,I\times\Omega_\gamma}^{2}}_{y,J_\beta,0}
+ \prod_{\gamma\in\{0,1\}}\vkqqq{\vkqq{F_3(x,y)}_{x,I\times\Omega_\gamma}^{2}}_{y,J_\beta,0}\Big) ,
\end{align*}
so that
\begin{equation}\label{eqbineq3}
|\mathcal{A}_3| \leq \mathcal{A}_4 + \mathcal{A}_5 .
\end{equation}
Combining (\ref{eqbineq1}),(\ref{eqbineq2}),(\ref{eqbineq3}) we finally get
$$ |\mathcal{A}-\Box\mathcal{B}_1| \leq \Box\mathcal{B}_{2}+{\textstyle\frac{1}{2}}\Box\mathcal{B}_{3}
+{\textstyle\frac{1}{2}}\Box\mathcal{B}_{4}+{\textstyle\frac{1}{2}}\Box\mathcal{B}_{5} , $$
which is what we needed to prove.
\end{proof}

\section{Completing the proof}

Let us introduce two auxiliary expressions
$$ \Xi_{k}^{+} := \sum_{I\times J\times \Omega\in\mathcal{T}_{k}}\!\!|I\times J|\ \,\mathcal{B}_{+}(I\times J\times \Omega) , \qquad
\Xi_{k}^{-} := \sum_{I\times J\times \Omega\in\mathcal{T}_{k}}\!\!|I\times J|\ \,\mathcal{B}_{-}(I\times J\times \Omega) $$
for $k=0,1,2,\ldots,M$.
By definition of the operator $\Box$ we have
\begin{align*}
\Xi_{k+1}^{\pm} - \Xi_{k}^{\pm}
& = \sum_{I\times J\times \Omega\in\mathcal{M}_{k}}\!\!\!
\Big(\sum_{\alpha,\beta\in\{0,1\}}\!\!|I_\alpha\!\times J_\beta|\ \mathcal{B}_{\pm}(I_{\alpha}\!\times J_{\beta}\!\times \Omega)
-\sum_{\gamma\in\{0,1\}}\!\!|I\times J|\ \mathcal{B}_{\pm}(I\times J\times \Omega_{\gamma})\Big) \\
& = \sum_{I\times J\times \Omega\in\mathcal{M}_{k}}\!|I\times J|\ (\Box\mathcal{B}_{\pm})(I\times J\times \Omega) .
\end{align*}
Summing for $k=0,1,\ldots,M\!-\!1$, using (\ref{eqarewritten}), and applying Lemma \ref{lmlemmaineq} we conclude
\begin{equation}\label{eqaestimated}
\Xi_{M}^{-} - \Xi_{0}^{-} \leq \Lambda_{\mathrm{W}} \leq \Xi_{M}^{+} - \Xi_{0}^{+} .
\end{equation}

It remains to estimate $\Xi_{M}^{\pm}$ and $\Xi_{0}^{\pm}$.
Recall that all cuboids $I\times J\times\Omega\in\mathcal{T}_{M}$ have the ``frequency interval'' $\Omega$ equal to $[0,2^M)$
and that $F_1,F_2,F_3$ are constant on the squares $I\times J$. Therefore,
\begin{align*}
\Xi_{M}^{\pm} & = \int_{[0,1)^2} \!\!F_1(x,y)F_2(x,y)F_3(x,y) \,dx dy \\
& \quad\pm\int_{[0,1)^2} \!\!\big( F_1(x,y)^2 + {\textstyle\frac{1}{2}}F_2(x,y)^2 F_3(x,y)^2
+ {\textstyle\frac{1}{2}}F_2(x,y)^4 + {\textstyle\frac{1}{2}}F_3(x,y)^4 \big) \,dx dy ,
\end{align*}
so H\"{o}lder's inequality and (\ref{eqnormalize}) imply
\begin{equation}\label{eqksi1}
|\Xi_{M}^{\pm}| \,\leq\, \|F_1\|_{\mathrm{L}^{2}} \|F_2\|_{\mathrm{L}^{4}} \|F_3\|_{\mathrm{L}^{4}}
+ \|F_1\|_{\mathrm{L}^{2}}^2 + {\textstyle\frac{1}{2}}\|F_2\|_{\mathrm{L}^{4}}^2 \|F_3\|_{\mathrm{L}^{4}}^2
+ {\textstyle\frac{1}{2}}\|F_2\|_{\mathrm{L}^{4}}^4 + {\textstyle\frac{1}{2}}\|F_3\|_{\mathrm{L}^{4}}^4 \,=\, {\textstyle\frac{7}{2}} .
\end{equation}
On the other hand, all cuboids in $\mathcal{T}_{0}$ have ``time intervals'' $I,J$ equal to $[0,1)$, so
{\allowdisplaybreaks\begin{align*}
\Xi_{0}^{\pm} & =
\sum_{n=0}^{2^M-1} \vkqqq{\vkqq{F_1(x,y)}_{x,[0,1),0}}_{y,[0,1),n}
\vkqqq{\vkqq{F_2(x,y)}_{x,[0,1),n} \vkqq{F_3(x,y)}_{x,[0,1),n}}_{y,[0,1),n} \\[-1mm]
& \quad\pm\sum_{n=0}^{2^M-1}
\Big(\vkqqq{\vkqq{F_1(x,y)}_{x,[0,1),0}}_{y,[0,1),n}^2
+ {\textstyle\frac{1}{2}}\vkqqq{\vkqq{F_2(x,y)}_{x,[0,1),n} \vkqq{F_3(x,y)}_{x,[0,1),n}}_{y,[0,1),n}^2 \\[-3.5mm]
& \qquad\qquad\quad + {\textstyle\frac{1}{2}}\vkqqq{\vkqq{F_2(x,y)}_{x,[0,1),n}^2}_{y,[0,1),0}^2
+ {\textstyle\frac{1}{2}}\vkqqq{\vkqq{F_3(x,y)}_{x,[0,1),n}^2}_{y,[0,1),0}^2 \Big) .
\end{align*}}
By (W2) and (W3),
{\allowdisplaybreaks\begin{align*}
\Xi_{0}^{\pm} & =
\int_{[0,1)^4} \!\!F_1(x_1,x_2\!\oplus\! x_3\!\oplus\! y) F_2(x_2,y) F_3(x_3,y) \,dx_1 dx_2 dx_3 dy \\
& \quad\pm\Big(
\int_{[0,1)^3} \!\!F_1(x_1,y) F_1(x_2,y) \,dx_1 dx_2 dy \\
& \qquad\quad + \frac{1}{2}\int_{[0,1)^5} F_2(x_1,y) F_2(x_2,x_1\!\oplus\! x_2\!\oplus\! x_3\!\oplus\! x_4\!\oplus\! y) \\[-3mm]
& \qquad\qquad\qquad\quad\ \ F_3(x_3,y) F_3(x_4,x_1\!\oplus\! x_2\!\oplus\! x_3\!\oplus\! x_4\!\oplus\! y)
\,dx_1 dx_2 dx_3 dx_4 dy \\
& \qquad\quad + \frac{1}{2}\int_{[0,1)^5} \!\!F_2(x_1,y_1) F_2(x_2,y_1) F_2(x_3,y_2)
F_2(x_1\!\oplus\! x_2\!\oplus\! x_3,y_2)\,dx_1 dx_2 dx_3 dy_1 dy_2 \\
& \qquad\quad + \frac{1}{2}\int_{[0,1)^5} \!\!F_3(x_1,y_1) F_3(x_2,y_1)
F_3(x_3,y_2) F_3(x_1\!\oplus\! x_2\!\oplus\! x_3,y_2)
\,dx_1 dx_2 dx_3 dy_1 dy_2 \Big) .
\end{align*}}
Once again, we can use H\"{o}lder's inequality and (\ref{eqnormalize}) to conclude
\begin{equation}\label{eqksi2}
|\Xi_{0}^{\pm}| \,\leq\, \|F_1\|_{\mathrm{L}^{2}} \|F_2\|_{\mathrm{L}^{4}} \|F_3\|_{\mathrm{L}^{4}}
+ \|F_1\|_{\mathrm{L}^{2}}^2 + {\textstyle\frac{1}{2}}\|F_2\|_{\mathrm{L}^{4}}^2 \|F_3\|_{\mathrm{L}^{4}}^2
+ {\textstyle\frac{1}{2}}\|F_2\|_{\mathrm{L}^{4}}^4 + {\textstyle\frac{1}{2}}\|F_3\|_{\mathrm{L}^{4}}^4 \,=\, {\textstyle\frac{7}{2}} .
\end{equation}
In the end, from (\ref{eqaestimated}),(\ref{eqksi1}),(\ref{eqksi2}) we deduce
$$ |\Lambda_{\mathrm{W}}(F_1,F_2,F_3)| \leq 7 , $$
which is exactly the estimate in Theorem \ref{thmmaintheorem}.

\end{document}